\documentclass[12pt]{article} 

\usepackage{amsmath}
\usepackage{amsthm}
\usepackage{amsfonts}
\usepackage{mathrsfs}
\usepackage{stmaryrd}
\usepackage{setspace}
\usepackage{fullpage}
\usepackage{amssymb}
\usepackage{breqn}
\usepackage{enumitem}
\usepackage{bbold} 
\usepackage{authblk}
\usepackage{comment}
\usepackage{hyperref}
\usepackage{pgf,tikz}
\usepackage{graphicx}
\usetikzlibrary{decorations.pathreplacing,arrows}
\tikzstyle{vertex}=[circle,draw=black,fill=black,inner sep=0,minimum size=3pt,text=white,font=\footnotesize]

\newtheorem*{thm*}{Theorem}
\newtheorem{thm}{Theorem}%[chapter]

\newtheorem{lemma}[thm]{Lemma}

\newtheorem{proposition}[thm]{Proposition}

\newtheorem{defin}{Definition}

\newtheorem{clm}[thm]{Claim}
\newtheorem*{proposition*}{Proposition}

\newcommand\cF{{\mathcal F}}

\newcommand\cH{{\mathcal H}}

\newcommand\cK{{\mathcal K}}
\newcommand\cL{{\mathcal L}}

\newcommand{\ignore}[1]{}

\title{Ramsey problems for Berge hypergraphs
}
\begin{document}

\author{D\'aniel Gerbner \thanks{Alfr\'ed R\'enyi Institute of Mathematics, Hungarian Academy of Sciences. e-mail: gerbner@renyi.hu}
\qquad
Abhishek Methuku \thanks{\'Ecole Polytechnique F\'ed\'erale de Lausanne. e-mail: abhishekmethuku@gmail.com}
\qquad
Gholamreza Omidi \thanks{Isfahan University of Technology and Institute for Research in Fundamental Sciences (IPM), Isfahan. e-mail: romidi@cc.iut.ac.ir}
\qquad
M\'at\'e Vizer \thanks{Alfr\'ed R\'enyi Institute of Mathematics, Hungarian Academy of Sciences. e-mail: vizermate@gmail.com.}}

\maketitle

\begin{abstract}
For a graph $G$, a hypergraph $\cH$ is a Berge copy of $G$ (or a Berge-$G$ in short), if there is a bijection $f : E(G) \rightarrow  E(\mathcal{H})$ such that for each $e \in E(G)$ we have $e \subseteq f(e)$. We denote the family of $r$-uniform hypergraphs that are Berge copies of $G$ by $B^rG$.

For families of $r$-uniform hypergraphs $\mathbf{H}$ and $\mathbf{H}'$, we denote by $R(\mathbf{H},\mathbf{H}')$ the smallest number $n$ such that in any red-blue coloring of the (hyper)edges of $\cK_n^r$ (the complete $r$-uniform hypergraph on $n$ vertices) there is a monochromatic blue copy of a hypergraph in $\mathbf{H}$ or a monochromatic red copy of a hypergraph in $\mathbf{H}'$. $R^c(\mathbf{H})$ denotes the smallest number $n$ such that in any coloring of the hyperedges of $\cK_n^r$ with $c$ colors, there is a monochromatic copy of a hypergraph in $\mathbf{H}$.

In this paper we initiate the general study of the Ramsey problem for Berge hypergraphs, and show that if $r> 2c$, then $R^c(B^rK_n)=n$. In the case $r = 2c$, we show that $R^c(B^rK_n)=n+1$, and if $G$ is a non-complete graph on $n$ vertices, then $R^c(B^rG)=n$, assuming $n$ is large enough. In the case $r < 2c$ we also obtain bounds on $R^c(B^rK_n)$. Moreover, we also determine the exact value of $R(B^3T_1,B^3T_2)$ for every pair of trees $T_1$ and $T_2$. 
\end{abstract}

\section{Introduction}

Ramsey's theorem states that for any graph $G$ (or any $r$-uniform hypergraph $\cH$) and integer $c$, there exists $N$ such that if we color each edge of the complete graph (or each hyperedge of the complete $r$-uniform hypergraph) on $N$ vertices with one of $c$ colors, then there is a monochromatic copy of $G$ (resp. $\cH$). This is the starting point of a huge area of research, see \cite{cfs} for a recent survey. Determining Ramsey numbers (the smallest integer $N$ with such a property) is a major open problem in combinatorics even for small particular graphs.

Let us introduce some basic definitions and notation. For graphs we denote by $K_n$ the complete graph on $n$ vertices, by $P_n$ the path on $n$ vertices and by $S_n$ the star with $n$ vertices.
A double star is a tree which has exactly two non-leaf vertices. We will often talk about graphs and hypergraphs as well. To make it easier to distinguish them, we use the word edge only in case of graphs, while we use the word hyperedge in case of larger uniformity.

We are going to deal with colorings of the hyperedges of $r$-uniform hypergraphs, $r$-graphs in short. A \textit{$c$-coloring} of a hypergraph is a coloring of its hyperedges with colors $1,2,\dots, c$. Each hyperedge gets exactly one color, but it is allowed that a color is not used at all. More precisely the coloring is a function $f$ from the set of hyperedges to the set $\{1,\dots, c\}$. We call a hypergraph $\cH$ with such a function $f$ a $c$-colored $\cH$ and denote it by $(\cH,f)$.
In the case of two colors sometimes we will call the colors blue and red to make it easier to follow the arguments, and call the 2-colored $\cH$ simply blue-red $\cH$. We will be interested in hypergraphs such that all their hyperedges are of the same color; we call them \textit{monochromatic}. In case we also know the color, say blue, we simply call a hypergraph which only has blue hyperedges \textit{monoblue}. We will denote the complete $r$-uniform hypergraph on $n$ vertices by $\cK_n^r$. 

For families of $r$-uniform hypergraphs $\mathbf{H_1},\mathbf{H_2},\dots,\mathbf{H_c}$, we denote by $R(\mathbf{H_1},\mathbf{H_2},\dots,\mathbf{H_c})$ the smallest number $n$ such that in any $c$-coloring of $\cK_n^r$, there is an $i\le c$ such that there is a monochromatic copy of a hypergraph in $\mathbf{H_i}$ of color $i$. If $\mathbf{H_1}=\mathbf{H_2}=\dots=\mathbf{H_c}=\mathbf{H}$, then $R(\mathbf{H_1},\mathbf{H_2},\dots,\mathbf{H_c})$ is denoted by $R^c(\mathbf{H})$.
%denotes the smallest number $n$ such that in any coloring of the hyperedges of $\cK_n^r$ with $c$ colors, there is a monochromatic copy of a hypergraph in $\mathbf{H}$.

The classical definition of a hypergraph cycle is due to Berge. Extending this definition, Gerbner and Palmer \cite{gp1} defined the following. For a graph $G$, a hypergraph $\cH$ is a \textit{Berge copy} of $G$ (a Berge-$G$, in short), if there is a bijection
$g : E(G) \rightarrow  E(\mathcal{H})$ such that for $e \in E(G)$ we have $e \subset g(e)$. In other words, $\mathcal{H}$ is a Berge-$G$ if we can embed a distinct graph edge into each hyperedge of $\mathcal{H}$ to create a copy of the graph $G$ on the vertex set of $\mathcal{H}$. We denote the family of $r$-uniform hypergraphs that are Berge copies of $F$ by $B^rF$. Extremal problems for Berge hypergraphs have attracted the attention of many researchers, see e.g. \cite{gmp,gmv,gp1,gp2,gykl,gyl,gymstv,pttw}. 

In this paper, we initiate the general study of Ramsey problems for Berge hypergraphs. We note that similar investigations have been started very recently and independently by Axenovich and Gy\'arf\'as \cite{agy} and by Salia, Tompkins,  Wang and Zamora \cite{STZZ}.  In \cite{agy} the authors focus on small fixed graphs where the number of colors may go to infinity. They also consider the non-uniform version. In \cite{STZZ} the authors focus mainly on the case of two colors.

Ramsey problems for Berge cycles have been well-studied. This line of research was initiated by Gy\'arf\'as, Lehel, S\'ark\"ozy and Schelp \cite{gylss}. They conjectured that $R^{r-1}(B^rC_n)=n$ for $n$ large enough, and proved it for $r=3$. Gy\'arf\'as and S\'ark\"ozy \cite{gys} proved $R^3(B^3C_n)=(1+o(1))5n/4$. Gy\'arf\'as, S\'ark\"ozy and Szemer\'edi \cite{gyssz} proved that $R^3(B^4C_n)\le n+10$ for $n$ large enough, and they proved $R^{r-1}(B^rC_n)=(1+o(1))n$ in \cite{gyssz2}. Maherani and Omidi \cite{mo} proved $R^3(B^4C_n)=n$ for $n$ large enough, and finally Omidi \cite{reza} proved the conjecture of  Gy\'arf\'as, Lehel, S\'ark\"ozy and Schelp \cite{gylss} by showing $R^{r-1}(B^rC_n)=n$ for $n$ large enough.  In this paper we are also mainly interested in $R(B^rF,B^rG)$ in the case when the number of vertices in both $F$ and $G$ are large enough.

A related problem is covering the vertices of a hypergraph by some monochromatic structures. This has also been considered in the Berge sense \cite{gmov,gys2, BCF}.

Note that $R^c(B^rF)$ is monotone decreasing in $r$, thus the known upper bounds for the graph case imply the same bounds for larger $r$. However, those bounds can be exponential in the cardinality of the vertex set of $F$. We believe that for $r\ge 3$ the situation is completely different and that the Ramsey number is always polynomial in $|V(F)|$. 

%Note that if $F$ is a subgraph of $G$, then $R^c(B^rF)\le R^c(B^rG)$, thus it is enough to state the conjecture for complete graphs.

%\begin{conjecture} For any fixed integers $r\ge 3$ and $c$, we have $R^c(B^rK_n)=n^{O(1)}$.

%\end{conjecture}

%We are able to give a linear upper bound if $c<r$, in particular, for two colors and any uniformity. Salia, Tompkins, Wang and Zamora  \cite{STZZ} give superlinear lower bounds in some other cases. However, their bounds are subquadratic.

\subsection*{New results}
%First we study Ramsey problems for Berge cliques. 
Our results are divided into three main types: $r > 2c$, $r = 2c$ and $ r < 2c$, as stated below. In the first two cases we almost completely resolve the problem.
%This part was confusing and I do not see how to make it clearer.
%by determining the Ramsey number for any $c$ $r$-uniform Berge hypergraphs in case the order of at least one of them is large enough.

\begin{proposition}
\label{r>2c}
Suppose $r> 2c$, and $n>c+c\binom{r}{2}$. Then $R^c(B^rK_n)=n$.
\end{proposition}

To state our result in case $r=2c$, we need a definition.
We say that a graph $G$ on at most $n$ vertices is \textit{good} if 
%why V_2 and V_3. Why not V_1 and V_2? what is the motivation of this obscure function?
its vertex set can be partitioned into two parts $V_1$ and $V_2$ with the following properties: $|V_1|=r-2$ and for any set $U\subset V_2$, the sum of the degrees of the vertices in $U$ is at most $|U|(n-r+2)-\binom{|U|}{2}$.
%$g(|U|):=(n-r+2)+(n-r+1)+\dots +(n-r+3-|U|)=\binom{n-r+3}{2}-\binom{n-r+3-|U|}{2}=|U|(n-r+2)-\binom{|U|}{2}$.

\begin{thm}\label{r = 2c} Let $r=2c\ge 4$ and $n\ge 12c(t-1)$. Let $n=|V(G_1)|\ge |V(G_2)|\ge \dots \ge |V(G_c)|$. Then 
\[ R(B^rG_1,\dots,B^rG_c)= \begin{cases} 
     n+1 & \textup{if one of the $G_i$s is not good and the remaining $G_i$ are $K_n$.} \\[1em]
      n & \textup{otherwise.}
   \end{cases}
\]

\end{thm}

% \begin{thm}
% \label{r = 2c}
% Suppose $r=2c\ge 4$ and $n$ is large enough. Then $R^c(B^rK_n)=n+1$.
% \end{thm}

% \begin{thm}\label{coroll} Suppose $r=2c\ge 4$ and $n$ is large enough, and let $G$ be a non-complete graph on $n$ vertices. Then $R^c(B^rG)=n$.
% \end{thm}
%{\color{red} We could have thm4 for $c$ different non-complete graphs with the same proof. For at least 2 non-complete and $c-2$ complete graphs we could have thm4 with small modification. 1 non-complete and $c-1$ complete graphs - I don't know. If the non-complete graph has more than $\binom{n-2c+2}{2}$ edges, then the answer is $n+1$, as the same lower bound works. }

Note that in the special case $c=2$, Proposition \ref{r>2c} and Theorem \ref{r = 2c} imply some of the results of Salia, Tompkins, Wang and Zamora  \cite{STZZ}.
%Did they actually prove it for general graphs for just for complete graphs K_n and K_{n+1} or something like this?

\begin{thm}
\label{r < 2c}
Suppose $r < 2c$. Then we have
\begin{enumerate}
\item [\textbf{(i)}] If $2<r$, then $1+c\lfloor\frac{n-2}{c-1}\rfloor\le R^c(B^rK_n)$.

\item [\textbf{(ii)}] If $c<r$, then $R(B^rK_{n_1}, \dots,B^rK_{n_c})\le \sum_{i=1}^c n_i$.

\item [\textbf{(iii)}]  If $c+1<r$ and $n>\frac{r-c-1}{c}(\binom{r}{2}+r-1)$, then $R^c(B^rK_{n}) \le \frac{c}{r-c-1}n$.

\item [\textbf{(iv)}] If $r=2c-1$ and $n>\frac{2c-3}{2c-1}(\binom{r}{2}+r-1)$, then $R^c(B^rK_{n}) \le (2c-1)\frac{n}{c-3}$.

\end{enumerate}
\end{thm}

So far all our upper bounds are linear in $n$.
In \cite{STZZ} the authors showed $R^3(B^3K_n)=\Omega(n^2/\log n)$ (note that the statement in \cite{STZZ} is actually weaker, but this is what follows from their proof). It is still possible that $R^c(B_rK_n)$ is polynomial in $n$ if $r\ge 3$. Here we show that even if it holds, the degree and coefficients of such a polynomial must depend on $c$.

\begin{proposition}\label{poly} $R^c(B^3K_n)= \Omega\left(\frac{n^{\lceil\frac{c}{2}\rceil}}{(\log n)^{\lceil c/2\rceil-1}}\right)$.

\end{proposition}

In case we have only two colors, the above theorems settle almost everything if the uniformity is at least 4. For 3-uniform hypergraphs, we prove the following.

\begin{proposition} 
\label{3-uniform} If $n\ge m>1$ and $n+m\ge 7$, then $n+m-3\le R(B^3K_n,B^3K_m)\le n+m-2$.
%We have
% \begin{enumerate}
% \item [\textbf{(i)}] 

% \item [\textbf{(ii)}] If $n\ge c$, then $R^c(B^3K_n)> (c-1)(n-2)$.
% \end{enumerate}
\end{proposition}

%We remark that if $n=m$ (and $n$ is large enough) then \textbf{(iv)} of Theorem \ref{r < 2c} implies that the lower bound in \textbf{(i)} of Proposition \ref{3-uniform} is the truth: $R(B^3K_n,B^3K_n)= 2n-3$.

Let us note that we use induction to prove the upper bound. The difference of one between the lower and upper bounds comes from the fact that the induction starts at small values $n$ and $m$ where $R(B^3K_n,B^3K_m)= n+m-2$ holds. One could easily eliminate this gap by dealing with additional small cases, potentially using a computer program. We have recently learned that the authors in \cite{STZZ} did exactly this, showing $R(B^3K_n,B^3K_m)= n+m-3$ if $n\ge 5$ and $m\ge 4$.

\vspace{2mm}

We also determine the exact value of the Ramsey number for every pair of Berge trees.

\begin{thm}\label{trees} Let $T_1$ and $T_2$ be trees with $n=|V(T_1)|\ge |V(T_2)|$. 
%Then we have %$$R(B^3T_1,B^3T_2)=n.$$

\[ R(B^3T_1,B^3T_2)= \begin{cases} 
     n+1 & \textup{if } |V(T_1)|= |V(T_2)|\le 4 \textup{ or } T_1=T_2=S_5 \\[1em]
      n & \textup{otherwise.}
   \end{cases}
\]
\end{thm}

%In this paper we determine $R^c(B^rG)$ for every $c\le r/2$ and for every $n$-vertex graph $G$ if $n$ is large enough. We also determine $R(B^3F,B^3G)$ if both $F$ and $G$ are trees.
%\smallskip
In case of $c=2$ and $r=3$ from Proposition \ref{3-uniform} we know that the Berge Ramsey number is between roughly $n$ and $2n$ for a graph $G$ on $n$ vertices. A parameter that might play an important role is the \textit{vertex cover number} of $G$, which is the smallest number of vertices such that every edge of $G$ is incident to at least one of them. Here we show that the Berge Ramsey number is close to $n$ for graphs with very small vertex cover number, but close to $2n$ for graphs with very large vertex cover number.

\begin{proposition}\label{vercov} If $G$ is a graph on $n$ vertices with vertex cover number $k\ge 3$, then $$2k-1\le R(B^3G,B^3G)\le n-k+2R(K_k,K_k).$$

\end{proposition}

The structure of the paper is as follows. In Section \ref{prel} we collect several lemmas that we will use later. In Section 3 we prove results in case $r\ge 2c$ (i.e., Proposition \ref{r>2c}, Theorem \ref{r = 2c}).
%, that determine every Ramsey number provided at least one of the forbidden Berge hypergraphs has a large enough vertex set. 
In Section 4 we prove results in the case $r<2c$ (Theorem \ref{r < 2c}, Proposition \ref{poly}, Proposition \ref{3-uniform}, Theorem \ref{trees} and Proposition \ref{vercov}). 
%Additionally, we determine the Ramsey number of every Berge hypergraph on enough vertices in the case the uniformity is at least twice the number of colors. 
%In Section 4 we focus on the case of two colors and in particular, we prove . 

\section{Preliminaries}\label{prel}

The \textit{shadow graph} of a hypergraph $\cH$ is the graph consisting of all the 2-edges that are subedges of a hyperedge of $\cH$.
We will often use the following auxiliary bipartite graph. Given a set $E_0$ of edges of the shadow graph, let $\Gamma(E_0)$ be the bipartite graph with part $A$ consisting of the edges in $E_0$, and part $B$ consisting of the hyperedges of $\cH$ containing the edges in part $A$, where an element of $B$ is connected to an element of $A$ if the corresponding hyperedge contains the edge. Let $\Gamma_i(E_0)$ denote the subgraph of $\Gamma(E_0)$ obtained if we delete from $B$ the hyperedges of color different from $i$.

Let $G$ be the shadow graph of a Berge-$F$-free $r$-graph $\cH$. The graph $G$ might contain a copy of $F$. Let $E_0$ be the set of edges in a copy of $F$ and consider $\Gamma(E_0)$. Observe that a matching covering $A$ in $\Gamma(E_0)$ would give a Berge copy of $F$ in $\cH$ by the definition of a Berge copy; a contradiction, thus there is no such matching.

%Consider the auxiliary bipartite graph with part $A$ consisting of the edges of a copy of $F$ in $G$, and part $B$ consisting of the hyperedges of $\cH$ containing the edges in part $A$, where an element of $B$ is connected to an element of $A$ if the corresponding hyperedge contains the edge. Observe that a matching covering $A$ would  give a Berge copy of $F$ in $\cH$ by the definition of a Berge copy; a contradiction.

Hall's marriage theorem states that a bipartite graph $G[A,B]$ does not contain a matching covering $A$ if and only if there is a subset $A'\subset A$ with $|N(A')|<|A'|$. We will use its following simple corollary several times, hence we state it as a lemma.

\begin{lemma}\label{philip} Let $\cH$ be an $r$-graph, and let $G$ be the graph formed by the 2-edges that are contained in at least $\binom{r}{2}$ hyperedges of $\cH$. Then if $G$ contains a copy of $F$, then $\cH$ contains a copy of Berge-$F$.

\end{lemma}

\begin{proof} Assume that $G$ contains a copy of $F$ and let $E_0$ be the set of its edges. A matching covering $A$ in $\Gamma(E_0)$ would give a copy of Berge-$F$ in $\cH$, as desired. So let us show that there is a matching covering $A$.

Suppose for a contradiction that there is no matching covering $A$. Then there is a set $A'\subset A$ with $|N(A')|<|A'|$ by Hall's theorem. The number of edges between $A'$ and $N(A')$ is at least $\binom{r}{2}|A'|$ by the construction of $G$, and at most  $\binom{r}{2}|N(A')|$, as an $r$-edge contains at most $\binom{r}{2}$ 2-edges. This contradicts $|N(A')|<|A'|$, finishing the proof.
\end{proof}

We will also use the following simple generalization.

\begin{lemma}\label{philip2} Assume that $F$ has the property that any set of $r$ vertices spans at most $t$ edges of $F$. Let $\cH$ be an $r$-graph, and let $G$ be the graph formed by the 2-edges that are contained in at least $t$ hyperedges of $\cH$. Then if $G$ contains a copy of $F$, then $\cH$ contains a copy of Berge-$F$.

\end{lemma}

\begin{proof} The proof goes the same way as the proof of Lemma \ref{philip}. We obtain $A'$ similarly. However, this time the number of edges between $A'$ and $N(A')$ is at least $t|A'|$ by the construction of $G$, and at most  $t|N(A')|$, as an $r$-edge contains at most $t$ 2-edges. This again contradicts $|N(A')|<|A'|$, finishing the proof.
\end{proof}

\begin{lemma}\label{feszfa} Let $T$ be a tree on $n$ vertices and $v\in V(T)$. Then there is a bijection $f:V(T)\setminus \{v\} \rightarrow E(T)$ such that $f(u)$ contains $u$ for every $u$.
\end{lemma}

\begin{proof} We prove it by induction on $n$, the base case $n=1$ is trivial. We obtain a forest $T'$ by deleting $v$ from $T$, let $T_1,\dots,T_k$ be its components, then each $T_i$ contains a neighbor $v_i$ of $v$. We apply the inductive hypothesis to each $T_i$ to find a bijection $f_i : V(T_i)\setminus \{v_i\} \rightarrow E(T_i)$. Let $f(x) = f_i(x)$ if $x \in V(T_i)\setminus \{v_i\}$, and $f(x) = vx$ if $x = v_i$ for some $i$. It is easy to see that $f$ is a bijection because each $f_i$ is a bijection which maps to edges of $E(T_i)$, and then we assigned only edges that do not belong to any $T_i$.
\end{proof}

\begin{defin}
Given an $r$-graph $\cH$ with vertex set $V(\cH)$ and a vertex $u \in V(\cH)$, the \textit{link hypergraph} $\cL_u$ is the $(r-1)$-graph on vertex set $V(\cH)$ consisting of the hyperedges $\{H\setminus \{u\}: u\in H\in \cH\}$. In case $r=3$, we call $\cL_u$ \textit{link graph}. If $\cH$ has a coloring, then $\cL_u$ has an \textit{inherited coloring}: the hyperedge $H\setminus \{u\}$ has the color of $H$.

\end{defin}

\begin{lemma}\label{ujfa} Let $T$ be a tree on $n\ge 6$ vertices. Then at least one of the following statements hold.

\textbf{(i)} There is a non-leaf vertex $v$ that is adjacent to exactly one non-leaf vertex $u$ such that deleting $v$ and its leaf neighbors we obtain a tree $T'$ that either has at least $6$ vertices or is a non-star on $5$ vertices.

\textbf{(ii)} There are two independent edges in $T$ such that any other edge is incident to at least one of their vertices. Moreover there are two adjacent edges with this property.

\textbf{(iii)} $T$ is a star or $P_6$.

\begin{proof} Let $T'$ be the tree we obtain if we remove all the leaves of $T$, and let $v$ be a leaf of $T'$. If $T$ is not a star, $T'$ has at least two vertices, in particular $v$ has a neighbor $u$ in $T'$. This shows $v$ is adjacent to exactly one non-leaf vertex in $T$. 

Observe that if $v$ has less than $n-5$ leaf neighbors in $T$, then deleting $v$ and its leaf neighbors from $T$, we obtain a tree $T''$ on at least five vertices. Then $T$ satisfies \textbf{(i)}, unless $T'' = S_5$, so $T$ is a double star. It is easy to see that double stars satisfy \textbf{(ii)}.

If $u$ is also a leaf in $T'$, then $T$ is a double star, so it satisfies \textbf{(ii)} again.
%the edges $uv$ and $vx$ are two adjacent edges, while the edges $uy$ and $vx$ are two independent edges showing \textbf{(ii)} holds, for any leaf neighbors $x$ of $v$ and $y$ of $u$.
Thus $u$ has another neighbor $w$ in $T'$, and $T'$ has a leaf $v'$ different from $v$ (note that $v'$ might be $w$). If $v$ or $v'$ has less than $n-5$ leaf neighbors in $T$, then we are done by the previous paragraph.

Thus $T$ has at least $2n-10$ leaves and three non-leaves, which implies $n\ge 2n-7$, i.e. $n\le 7$. Moreover, if $n=7$, we know $T'$ is a path on three vertices $v$, $u$ and $w$, and both $v$ and $w$ have two leaf neighbors. It is easy to see that this tree satisfies \textbf{(ii)}. It is also easy to see that all trees on $6$ vertices but the star and $P_6$ satisfy \textbf{(ii)}.
\end{proof}

\end{lemma}

\begin{lemma}\label{eni} For every set of positive integers $n_i$ ($1 \le i \le c$), we have $$\prod_{i=1}^cn_i>(\sum_{i=1}^c n_i)-c.$$

\end{lemma}

\begin{proof} We prove the lemma by induction on $c$. For $c=1$ the statement is trivial. Let $n_c$ be the smallest of the integers $n_i$ ($1 \le i \le c$). If $n_c=1$, then we have $$\prod_{i=1}^cn_i=\prod_{i=1}^{c-1}n_i>(\sum_{i=1}^{c-1} n_i)-(c-1) =(\sum_{i=1}^c n_i)-c.$$ 

If $n_c\ge 2$, then $$\prod_{i=1}^cn_i\ge 2\prod_{i=1}^{c-1}n_i>2(\sum_{i=1}^{c-1} n_i)-2(c-1)
=  (\sum_{i=1}^{c-1} n_i) + (\sum_{i=1}^{c-2} n_i)+n_{c-1}-2(c-1)$$ $$\ge (\sum_{i=1}^{c} n_i) + (\sum_{i=1}^{c-2} n_i)-2(c-1)\ge
(\sum_{i=1}^{c} n_i) + 2(c-2)-2(c-1)\ge (\sum_{i=1}^c n_i)-c.$$
Note that in the above inequalities we use that $n_{c-1}\ge n_c$, $n_i\ge 2$ ($1 \le i \le c$) and induction.
\end{proof}
%\ge2(\sum_{i=1}^{c-1} n_i)-c\ge (\sum_{i=1}^{c-1} n_i)-c.$$

\section{Large uniformity}

%\section{Proofs for Complete graphs}
\subsection{The case $r > 2c$. Proof of Proposition \ref{r>2c}}
Recall that Proposition \ref{r>2c} states that for $r>2c$ and $n>c+c\binom{r}{2}$ we have $R^c(B^rK_n)=n$.
Let us consider a $c$-colored $\cK_n^r$ with vertex set $V$ such that $|V|=n>c+c\binom{r}{2}$. Then we color an edge $uv$ ($u,v\in V$) with color $i$ if it is contained in at least $\binom{r}{2}$ hyperedges of color $i$ (thus an edge can get multiple colors). If there is a $K_n$ in the resulting graph $G$ such that all its edges are of color $i$, then that gives us a monochromatic Berge-$K_n$ of color $i$ by Lemma \ref{philip}.

Hence we can assume that for every color $i$, there is an edge $u_iv_i$ that is not of that color. Let us consider the set $U=\{u_1,\dots,u_c,v_1,\dots,v_c\}$. Obviously we have $|U|\le 2c < r$, thus $U$ is contained by at least $n-|U| \ge n - 2c >c\binom{r}{2}-c$ hyperedges. Thus there is a color $i$ shared by at least $\binom{r}{2}$ of those hyperedges, hence $u_iv_i$ is contained by at least $\binom{r}{2}$ hyperedges of color $i$, a contradiction.

\subsection{The case $r = 2c$. Proof of Theorem \ref{r = 2c}}

%suppose we are given a $c$-colored $\cK_{n+1}^r$, where $n$ is large enough. We give color $i$ to an edge $uv$ of the shadow graph if it is contained in at least $\binom{r}{2}$ hyperedges of color $i$. If there is an $i$ such that only one edge is not of color $i$, we delete one of its vertices and then there is a $K_n$ of color $i$ in the remaining part of the shadow graph, thus we have a monochromatic Berge-$K_n$ by Lemma \ref{philip}. If there are at least two edges not of color $i$ for every $i$, then we use Lemma \ref{r2c} (stated below) to find a monochromatic Berge-$K_{n}$ (in fact, in this case we can find a monochromatic Berge-$K_{n+1}$), which would finish the proof of the theorem. It remains to prove Lemma \ref{r2c}.

First we prove a lemma.
%that will easily imply Theorem \ref{coroll}. Then we prove Theorem \ref{r = 2c} using Theorem \ref{coroll}.

\begin{lemma}\label{r2c} Let $r=2c$ and $n\ge 12c\binom{r}{2}$ be given integers and consider a $c$-colored complete $r$-uniform hypergraph on a vertex set $V$ with $|V|=n$. Let us color an edge $uv$ (where $u,v\in V$) with color $i$ if $u$ and $v$ are contained together in at least $t=\binom{r}{2}+1$ hyperedges of color $i$. Note that this way an edge can get multiple colors. For every $i\le c$, let $E_i$ be the set of edges that do not have color $i$. If there are $j\neq l\le c$ with $|E_j|\ge 2$, $|E_l|\ge 2$, then we can find a monochromatic Berge-$K_n$.

\end{lemma}

\begin{proof}
%we start with an argument that immediately proves the statement for large enough $c$. We start the proof as in \textbf{(i)} by giving color $i$ to an edge $uv$ if it is contained in at least $t=\binom{r}{2}$ hyperedges of color $i$. If there is a $K_n$ of color $i$ in the resulting graph, then we have a monochromatic Berge-$K_n$ by Lemma \ref{philip}. 

%For every $i\le c$, let $E_i$ be the set of edges that do not have color $i$, then we have $|E_i|\ge 2$ by our assumption.  
Let $V_i$ be the set of vertices incident to an edge in $E_i$. Let us pick an edge from $E_i$ for every $i\le c$. If these $c$ edges together cover at most $2c-1<r$ vertices, then they are contained in more than $n-r>c(t-1)$ hyperedges. Thus more than $t$ of these hyperedges have the same color $i$, which leads to a contradiction. Therefore, we can assume that the $V_i$s are pairwise disjoint and if we pick one edge from each $E_i$, then these edges span an $r$-set.  Let $\cH$ be the hypergraph consisting of all the $r$-sets that can be obtained this way. (Since $\cH$ is a subhypergraph of the complete $r$-uniform hypergraph, the hyperedges of $\cH$ are colored with the $c$ colors as well). Observe that for every $i\le c$, a hyperedge in $\cH$ contains exactly two vertices from $V_i$ and exactly one edge from $E_i$.

%how do you compute c \ge 14 below? Are you using r =3 for lower bounding t?
%Let $l$ be the color that appears most often among the colors of the hyperedges of $\cH$, which means there are at least $(\prod_{i=1}^c|E_i|)/c\ge |E_l|2^{c-1}/c$ hyperedges of color $l$ in $\cH$. Each of these hyperedges contains exactly one edge from $E_l$, thus there is an edge in $E_l$ that is contained in at least $2^{c-1}/c$ hyperedges of color $l$. This contradicts the definition of $E_l$ if $2^{c-1}/c>t = \binom{2c}{2}$, i.e. $c\ge 14$, finishing the proof in that case. Moreover, for any $c$, if there is an $l'\neq l$ with $|E_{l'}|>\frac{ct}{2^{c-2}}$, then by the same argument there are at least $(\prod_{i=1}^c|E_i|)/c\ge |E_l||E_{l'}|2^{c-2}/c>|E_l|t$ hyperedges of color $l$ in $\cH$, thus there is an edge in $E_l$ that is contained in more than $t$ hyperedges, a contradiction. This gives an upper bound on the size of every $E_i$ when $i \not = l$ (note that this upper bound does not depend on $n$).

%we can replace $2^{c-1}$ by $|E_{l'}|2^{c-2}$ and arrive to a contradiction similarly.

Let us assume that for every $i\le c$ there is an edge $u_iv_i\in E_i$ that is not contained in a hyperedge in $\cH$ of color $i$. Then the hyperedge containing all such edges $u_iv_i$ (with  $i\le c$) is in $\cH$ and no matter what its color is, we have a contradiction. Hence we can assume that every edge in, say, $E_1$ is contained in at least one hyperedge of color $1$ in $\cH$. For $uv\in E_1$ let $f_0(uv)$ be one of those hyperedges of color $1$ in $\cH$ that contain $u$ and $v$. Note that $f_0$ is an injection from $E_1$ to the hyperedges of color $1$ (in $\cH$).

% \textbf{Case 1.} $|E_1|<\frac{c\binom{2c}{2}}{2^{c-2}}$. Then we delete $E_1$ together with the set $\{f_0(e): e\in E_1\}$ of hyperedges.  We claim that the resulting hypergraph $\cF$ has the property that every edge not in $E_1$ is contained in at least $t$ hyperedges, thus Lemma \ref{philip} finishes the proof. Let $V_i$ be the set of vertices incident to edges in $E_i$, $V'=\bigcup_{i=2}^c V_i$ and $k=|V'|$.  Let us consider first an edge $e=uv\not\in E_1$ that is incident to a vertex not in any $V_i$. Then $e\not\in E_1$, thus it is contained in at least $t$ hyperedges of color $1$. Those hyperedges are not in $\cH$, thus they are not deleted, they are in $\cF$, finishing this case. If $u,v\in V_1$ but $uv\not \in E_1$, then the same proof holds. Otherwise at least one of $u$ and $v$ is in $V'$, thus we can pick an edge from every $E_i$ such that their union with $u$ and $v$ has size at most $2c-1$, which means they are contained in at least $n-2c+1$ hyperedges. At most $(c-1)(t-1)$ of those can be of color different from $1$, which finishes the proof for $n$ large enough.
% \textbf{Case 2.} $|E_1|\ge \frac{c\binom{2c}{2}}{2^{c-2}}$. This implies $l=1$ and for $i>1$ we have $|E_i|<\frac{c\binom{2c}{2}}{2^{c-2}}$. In particular $k\le (c-1)c\frac{\binom{2c}{2}}{2^{c-2}}$. 

%say we extend?
Let us sketch briefly the next part of the proof. Our goal would be to extend $f_0$ to those edges that are contained in many hyperedges of color 1. However, even in that case it is possible that all those hyperedges are already images of $f_0$. Therefore, we also have to change what the edges in $E_1$ are mapped into.
We will build an injection $f$ from a larger set of the edges to hyperedges of color $1$ (in $\cH$) containing them. Note that there is no connection between $f_0$ and $f$.

Let $V'=\bigcup_{i=2}^c V_i$. If we pick an edge from every $E_i$ with $i>1$ then they span a set of $(r-2)$ vertices by above. Let $\cH'$ be the $(r-2)$-uniform hypergraph consisting of all such $(r-2)$-sets. Note that $\cH'$ has at least two hyperedges since at least one of the $E_i$'s with $i\neq 1$ contains at least two edges (this is where we use the assumption that at least two of the $E_i$s have size at least 2).

We call an edge $uv$ with $u\in V_1$, $v\in V'$ or $u,v\in V'$ \textit{nice} if it is contained in at least $p=n/2$ hyperedges in $\cH$ of color $1$. Consider $\Gamma_1(E_1)$. Obviously every vertex in $B$ has degree at most $t-1=\binom{r}{2}$. %check.

\begin{clm}\label{klem1} Let $A'\subset E_1$ be of size $x$. Then the neighborhood of $A'$ in $\Gamma_1(E_1)$ has size at least $2x-2t(c-1)$.

\end{clm}
%why fix only H and H'
\begin{proof} Consider two sets $H,H'$ from $\cH'$. There are $2x$ hyperedges of $\cH$ of the form $H \cup e$ or $H' \cup e$ for $e \in A'$, but these hyperedges do not necessarily have color $1$. For each $i>1$, let us fix an edge of $E_i$ contained in $H$ and similarly an edge of $E_i$ contained in $H'$. This shows that there are at most $2t$ hyperedges of color $i$ (for each $i>1$) containing either $H$ or $H'$ altogether. Thus, out of the $2x$ hyperedges of $\cH$ of the form $H \cup e$ or $H' \cup e$ for $e \in A'$, at least $2x-2t(c-1)$ have color $1$, and these hyperedges are obviously in the neighborhood of $A'$ in $G$.
\end{proof}

\begin{clm} There is a matching covering $A$ in $\Gamma_1(E_1)$.
\end{clm}

\begin{proof} Suppose indirectly that there is no such matching, then by Hall's theorem there is a set $S\subset A$ with $|N(S)|<|S|$. Let $S_1=S\cap E_1$ and $S_2=S\setminus S_1$, $x=|S_1|$ and $y=|S_2|$. Then the function $f_0$ shows $|N(S_1)|\ge |S_1|$, thus $S_2$ is non-empty. Let $e\in S_2$. By the definition of nice edges $e$ has at least $p$ neighbors. This implies $x+y=|S|>|N(S)|\ge |N(S_2)|\ge p$. %check py?

%Observe that any element of of $B$ has exactly one neighbor in $E_1$ and two elements of $E_1$ do not have common neighbors in $G$. Thus we find $p$ elements of $E_1$

By Claim \ref{klem1} we have $|N(S)|\ge |N(S_1)|\ge 2x-2t(c-1)$, which implies $y\ge x-2t(c-1)$ since $x+y=|S|>|N(S)|$. Now the number of edges between $S$ and $N(S)$ is at least $|S_1|+p|S_2|=x+py$, while on the other hand it is at most $t|N(S)|<t(x+y)$ (as every vertex in $N(S)$ has degree at most $t$). Rearranging $x+py<t(x+y)$, and using $t = \binom{r}{2}+1$ and $p=n/2$, we obtain that $$y<\frac{x\binom{r}{2}}{\frac{n}{2}-\binom{r}{2}-1}<\frac{x}{2}.$$ Thus we have $$x-2t(c-1)\le y<x/2,$$ i.e. $x<4t(c-1)$, which also implies $y<2t(c-1)$. Hence we have $6t(c-1)>x+y>p=n/2$, a contradiction with our assumption on $n$.
\end{proof}

%The above inequalities together imply $2x-2t(c-1)<\frac{x(\binom{r}{2}-2)-2t(c-1)}{p-\binom{r}{2}}$, i.e. $x(2p-\binom{r}{2}-2)<2t(c-1)(p-\binom{r}{2}-1)$... ((thus $x$ is small, but $y$ is too by the same calculations done smarter, but $x+y>p$))

% We want to find matching covering $A$ here. If we can, we are done, as not nice edges have many hypes of color 1 not in $\cH$. Suppose there is a blocking set $S$, with $x=|S_1|$ elements $S_1$ from $E_1$ and $y=|S_2|$ others. Then $y>0$, that element $e$ has $p$ neighbors, thus $|S|>|N(S)|\ge p$. Things from $S_1$ don't have common neighbors. The $p$ neighbors of $e$ are all connected to one element from $E_1$, and at most constant $q$ of those have degree one, thus $p-q$ have degree $2$, hence $|N(S)|\ge p+|S_1|-q$. Thus $|S_2|\ge p-q$, there are $p(p-q)$ edges between $S$ and $N(S)$, thus $|N(S)|\ge p(p-q)/t$. Hence $S_1$ is large, but then $N(S_1)$ is also large.

Let $e$ be an element of $A$, i.e., a nice edge or an edge in $E_1$. Then we denote by $f(e)$ the  hyperedge that is matched to $e$ in $\Gamma_1(E_1)$. Let us delete the set of hyperedges $\{f(e): e\in A\}$ from the $c$-colored $\cK_n^r$. Let $\cF$ denote the remaining (colored) hypergraph.

\begin{clm}
Every edge not in $A$ is contained in at least $t$ hyperedges of color $1$ in  $\cF$.
\end{clm}

\begin{proof}
Let us consider an edge $uv\not \in A$. Then it is originally contained in at least $t$ hyperedges of color $1$ (in the $c$-colored $\cK_n^r$ before the set of hyperedges $\{f(e): e\in A\}$ were deleted). If $u$ is not in any $V_i$, or if $u,v\in V_i$ but $uv\not \in E_i$, then no hyperedge of $\cH$ contains $uv$, thus $\cF$ contains the same $t$ hyperedges of color $1$.
Otherwise we can pick an edge from every $E_i$ ($i>1$) such that the set $S$ of vertices spanned by these edges contains at least one of the vertices $u$, $v$. Then the set $S \cup \{u,v\}$ has size at most $2c-1$, thus it is contained in at least $n-2c+1$ hyperedges. At most $(c-1)(t-1)$ of these hyperedges have color different from $1$, thus at least $n-2c+1-(c-1)(t-1) \ge p+t$ of them have color $1$. As $uv$ is not nice, less than $p$ hyperedges containing $uv$ appear as $f(e)$ for some $e \in A$. Thus at least $t$ hyperedges of color $1$ containing $uv$ are in $\cF$, proving the claim.

\end{proof}

Now we can find a Berge copy of the graph consisting of all the edges not in $A$ in color $1$, applying Lemma \ref{philip}. Then we represent each edge $e\in A$ with $f(e)$ to obtain a Berge-$K_n$ in color $1$, finishing the proof.

\end{proof}

Using Lemma \ref{r2c}, we will now prove Theorem \ref{r = 2c}. We restate Theorem \ref{r = 2c} below for convenience.

Recall that a graph $G$ on at most $n$ vertices is \textit{good} if 
%it has at most $\binom{n-r+2}{2}$ edges and 
its vertex set can be partitioned into two parts $V_1$ and $V_2$ with $|V_1|=r-2$ such that for any set $U\subset V_2$ the sum of the degrees of the vertices in $U$ is at most $g(|U|) := |U|(n-r+1)-\binom{|U|}{2}$. Otherwise $G$ is not good.
%Abhi: Here maybe we should note that this number counts something? or some motivation? 
\begin{thm*} Let $r=2c\ge 4$ and $n$ be large enough. Let $n=|V(G_1)|\ge |V(G_2)|\ge \dots \ge |V(G_c)|$. Then 
\[ R(B^rG_1,\dots,B^rG_c)= \begin{cases} 
     n+1 & \textup{if one of the $G_i$s is not good and the remaining $G_i$s are $K_n$}, \\[1em]
      n & \textup{otherwise.}
   \end{cases}
\]
\end{thm*}

%\begin{thm*} Let $G$ be a non-complete graph on $n$ vertices, where $n$ is large enough, and $2c=r\ge 4$. Then $R^c(B^rG)=n$.
%\end{thm*}

\begin{proof} The lower bound $n$ is trivial. For the lower bound $n+1$ in the appropriate cases, let us consider the complete $r$-graph on $n$ vertices and let $u_1,v_1, \dots$,$u_{c-1},v_{c-1}$ be $2c-2$ distinct vertices. Let us color the hyperedges containing all of these vertices with color $c$. For every other hyperedge $H$ there is an $i\le c-1$ such that $H$ contains at most one of the vertices $u_i, v_i$. Then let the color of $H$ be the smallest such $i$. This way we colored all the hyperedges of the complete $r$-graph on $n$ vertices. For any $i<c$, the edge $u_iv_i$ is not contained in any hyperedge of color $i$, thus there is no monochromatic Berge-$K_n$ of color $i$. 

%Let $\cH$ be the hypergraph consisting of the hyperedges of color $c$. 
Let $V_1=\{u_1,v_1, \dots$,$u_{c-1},v_{c-1}\}$ and $V_2$ be the set of the remaining vertices. Consider an arbitrary $U\subset V_2$. Every vertex of $U$ is incident to $n-r+1$ hyperedges of color $c$. The number of hyperedges of color $c$ incident to vertices in $U$ is at most $|U|(n-r+1)-\binom{|U|}{2} =: g(U)$, as for every pair of vertices $u,v\in U$, we counted the hyperedge consisting of $u$, $v$ and $V_1$ twice. This shows that the hypergraph consisting of the hyperedges of color $c$ can only be a Berge copy of good graph, finishing the proof of the lower bound.

\vspace{2mm}

Below we prove the corresponding upper bounds. 

First let us consider the case when one of the $G_i$s is not good and the remaining $G_i$s are equal to $K_n$. Our goal is to show the upper bound $n+1$ in this case. Let us consider a $c$-colored complete $r$-uniform hypergraph on a vertex set of size $n+1$. We define $E_i$ as in Lemma \ref{r2c}. If there are two $E_i$s of size more than one, we find a monochromatic Berge-$K_{n+1}$ by Lemma \ref{r2c}. If there is an $i$ such that $|E_i|=1$, then by the definition of $E_i$, we know that all but one edge of $K_{n+1}$ are contained in at least $\binom{r}{2}+1$ hyperedges of color $i$. In other words, if $K_{n+1}^{-}$ denotes the graph obtained by removing exactly one edge from $K_{n+1}$, then Lemma \ref{philip} implies that we can find a copy of Berge-$K_{n+1}^{-}$ in color $i$, which of course, contains a copy of Berge-$K_n$. This proves the upper bound $n+1$ (in all the cases of the theorem).

We will now show that a (better) upper bound $n$ holds in the remaining cases of the theorem. Let us consider a $c$-colored complete $r$-uniform hypergraph on a vertex set $V$ of size $n$. We define $E_i$ as in Lemma \ref{r2c}. If there are two $E_i$s of size more than one, then we find a monochromatic Berge-$K_{n}$ by Lemma \ref{r2c}. Hence at least $c-1$ of the $E_i$s have size one. If $|E_i|=1$ for some $i$, then by the definition of $E_i$, all but one edge of $K_{n}$ are contained in at least $\binom{r}{2}+1$ hyperedges of color $i$. So Lemma \ref{philip} implies that we can find a copy of Berge-$K_n^-$ in color $i$ (where $K_n^-$ denotes the graph obtained by removing one edge from $K_n$).
%we can find a Berge copy of $K_{n}$ minus an edge in color $i$ by Lemma \ref{philip}. 
Thus we are done unless there is at most one non-complete graph among the $G_i$s, say $G_c$, without loss of generality. Then $G_c$ must be a good graph, otherwise we are in the case that was already handled by the previous paragraph.
%the theorem states the upper bound $n+1$ which we have already proved.

Let $u_iv_i$ be the only element of $E_i$ for each $i<c$. Assume first that there is a hyperedge $H$ of color $i$ containing $u_iv_i$. Then we find a Berge-$K_n$ of color $i$ as follows. First we map $u_iv_i$ to $H$. Then every other edge of $K_n$ is contained in at least $t-1=\binom{r}{2}$ hyperedges of color $i$ different from $H$, thus by Lemma \ref{philip} we can map those edges to hyperedges of color $i$ different from $H$. Therefore, we can assume that $u_iv_i$ is not contained in any hyperedges of color $i$.

Let $U_1=\{u_1,v_1, \dots$,$u_{c-1},v_{c-1}\}$. Let $U_2$ be the set of the remaining vertices and let $\cH$ be the hypergraph consisting of the hyperedges containing $U_1$. Then the hyperedges of $\cH$ are all of color $c$. We will show that $\cH$ contains a Berge copy of every good graph. %on how many vertices?

Let $G$ be a good graph. We will embed its edges into distinct hyperedges of $\cH$. By definition, the vertices of $G$ are partitioned into $V_1$ and $V_2$; we consider an arbitrary bijection $\alpha$ that maps $V_1$ into $U_1$ and $V_2$ into $U_2$. We will build an embedding of $G$ in three steps. In the first step we embed the edges $uv$ of $G$ inside $V_2$, and we let $f(uv)=\{\alpha(u),\alpha(v)\}\cup U_1$. Let $\cH'$ be the subhypergraph of $\cH$ consisting of the hyperedges not of this form.

Then in the second step we embed the \textit{crossing} edges -- i.e., edges $uv$ such that $u\in V_1$, $v\in V_2$. Let $E'$ denote the set of edges in the shadow graph of $\cH'$ that have the form $\alpha(u)\alpha(v)$, where $uv$ is a crossing edge. We consider $\Gamma(E')$.
%the auxiliary bipartite graph where one part $A$ consists of the crossing edges and the other part $B$ consists of the hyperedges of $\cH'$, where an element $a=uv\in A$ is connected to an element $b\in B$ if $b$ contains $\alpha(u)$ and $\alpha(v)$. 
A matching covering $A$ here would mean we can extend the injection $f$ to the crossing edges. If there is no such matching, then by Hall's theorem there is a set $S$ of crossing edges $uv$ such that their images $\alpha(u)\alpha(v)$ are contained in less than $|S|$ hyperedges of $\cH'$ altogether. Let $V_0$ be the set of vertices in $V_2$ incident to at least one edge in $S$, and let $U_0=\alpha(V_0)$ (i.e., the image of $V_0$ under the map $\alpha$). Let $E_0$ be the set of the edges in $G$ incident to $V_0$, then $S\subseteq E_0$ and $|E_0|\le g(|U_0|)$. Let $\cH_0$ consist of hyperedges in $\cH$ incident to $U_0$, then $|\cH_0|=g(|U_0|)$. In the first step we mapped at most $g(|U_0|)-|S|$ edges of $E_0$ into hyperedges of $\cH$. Observe that no other edge is mapped into a hyperedge in $\cH_0$, thus there are at least $|S|$ hyperedges of $\cH_0$ in $\cH'$. We claim that each of these hyperedges contain an edge from $S$. Indeed, they each contain a vertex in $U_0$, thus an endpoint of an edge in $S$, and they each contain the other endpoint of that edge, since that is in $U_1$, which is contained in every hyperedge of $\cH$. This contradicts our assumption that the edges of $S$ are contained in less than $|S|$ hyperedges of $\cH'$ altogether.

Thus we have an injection $f'$ from the edges of $G$ inside $V_2$ and the crossing edges to distinct hyperedges in $\cH$ containing them. In the third step we are going to embed the remaining edges of $G$ (those inside $V_1$). Observe that each of them is contained in \emph{every} hyperedge of $\cH$. As the number of edges in $G$ is at most the number of hyperedges in $\cH$, we can choose a distinct remaining hyperedge of $\cH$ for every remaining edge of $G$ and we are done.
\end{proof}
\section{Small uniformity}

\subsection{The case $r < 2c$. Proof of Theorem \ref{r < 2c}}
%Abhi: how to explain the sentences like "Berge-K_n has to contain two vertices from V_i" We mean base vertices or core vertices?
We restate Theorem \ref{r < 2c} below for convenience.

\begin{thm*}
Suppose $r < 2c$. Then we have
\begin{enumerate}
\item [\textbf{(i)}] If $2<r$, then $1+c\lfloor\frac{n-2}{c-1}\rfloor\le R^c(B^rK_n)$.

\item [\textbf{(ii)}] If $c<r$, then $R(B^rK_{n_1}, \dots,B^rK_{n_c})\le \sum_{i=1}^c n_i$.

\item [\textbf{(iii)}]  If $c+1<r$ and $n>\frac{r-c-1}{c}(\binom{r}{2}+r-1)$, then $R^c(B^rK_{n}) \le \frac{c}{r-c-1}n$.

\item [\textbf{(iv)}] If $r=2c-1$ and $n>\frac{2c-3}{2c-1}(\binom{r}{2}+r-1)$, then $R^c(B^rK_{n}) \le (2c-1)\frac{n}{2c-3}$.

\end{enumerate}
\end{thm*}

To prove \textbf{(i)} we take a complete $r$-graph on $c\lfloor\frac{n-2}{c-1}\rfloor$ vertices.  We partition its vertex set into $c$ parts $V_1,\dots, V_c$, each of size $\lfloor\frac{n-2}{c-1}\rfloor$. For every hyperedge $H$, there is (at least) one part $V_i$ that $H$ intersects in at most one vertex since $r < 2c$. Then let the smallest such $i$ be the color of $H$. A Berge-$K_n$ of color $i$ has to contain at least two vertices $u,v$ from $V_i$, as the union of the other parts has size at most $n-2$. But there is no hyperedge of color $i$ containing $u,v$ (so the pair $u,v$ cannot be represented by a hyperedge of color $i$), a contradiction to our assumption that the Berge-$K_n$ of color $i$ contains the vertices $u,v$.

\smallskip
Now we prove \textbf{(ii)}. We use induction on $\sum_{i=1}^c n_i$, the cases when every $n_i$ is at most 2 are trivial as $r\le 2c$. Let us consider a $c$-colored complete $r$-graph on $\sum_{i=1}^c n_i$ vertices. We set aside a vertex $u$, then by induction there is a $B^rK_{n_i-1}$ of color $i$ on a subset $A_i$ of the remaining vertices, for every $i\le c$. If we can extend the $B^rK_{n_i-1}$ of color $i$, by adding the vertex $u$ and distinct hyperedges of color $i$ containing $uv$, for every $v\in A_i$, then we found the desired monochromatic $B^rK_{n_i}$ of color $i$. So we can assume for every $i\le c$, we cannot extend the $B^rK_{n_i-1}$ of color $i$ by $u$. Then by Hall's theorem, for every $i \le c$, there is a subset $B_i$ of vertices in $A_i$, such that the number of hyperedges of color $i$ containing $u$ and a vertex from $B_i$ is less than $|B_i|$.

Consider the ($(r-1)$-uniform) link hypergraph $\cL_u$ with the inherited coloring.
Let $\cH$ be the subfamily of $\cL_u$ consisting of the hyperedges of $\cL_u$ that intersect every $B_i$. 

\begin{clm}\label{newclaim}
$$|\cH|\ge 1+\sum_{i=1}^c |B_i|-c.$$
\end{clm}
\begin{proof}

If there is no vertex belonging to two different $B_i$s, then $\cH$ contains at least $\prod_{i=1}^c |B_i|$ hyperedges. Indeed, if we pick a vertex from every $B_i$, there is a hyperedge in $\cH$ containing them, as $r-1\ge c$. Moreover, there is such a hyperedge that contains no other vertices from $\bigcup_{i=1}^cB_i$, as $|\bigcup_{i=1}^cB_i|\le \sum_{i=1}^c(n_i-1)$, thus there are at least $c$ other vertices and $2c>r$. Such hyperedges of $\cH$ are counted only once when we pick one vertex from every $B_i$ $\prod_{i=1}^c |B_i|$ ways, thus $\cH$ contains at least $\prod_{i=1}^c |B_i|$ hyperedges.
In this case Lemma \ref{eni} finishes the proof of this claim.

If there is a vertex $v\in B_1\cap B_2$, let us pick a vertex $v_i\in B_i$ for $3\le i\le c$. There are at least $(\sum_{i=1}^c n_i)-1-(c-1)$ hyperedges in $\cL_u$ containing these at most $c-1$ vertices, and they all belong to $\cH$. This finishes the proof as $|B_i|<n_i$.
%Note that if $(\sum_{i=1}^c n_i)-1-(c-1)< 1+\sum_{i=1}^c |B_i|-c$, then $|B_i|=n_i$ for every $i\le c$, in particular $|B_1|>1$. This implies that we can find another hyperedge in $\cH$, finishing the proof.

\end{proof}

By Claim \ref{newclaim} we have $|\cH|>\sum_{i=1}^c (|B_i|-1)$, hence there is a color $i$ such that $\cH$ contains at least $|B_i|$ hyperedges of color $i$, a contradiction.
% we pick a vertex from every other $B_i$ and an arbitrary vertex. DANI: we count some multiple times, we have to be more precise!
% If there is no such vertex, then $\cH$ contains exactly $\prod_{i=1}^c n_i>\sum_{i=1}^c n_i-c$ hyperedges and we are done.

\smallskip

To prove $\textbf{(iii)}$, assume $n$ is large enough and consider a $c$-colored complete $r$-graph with vertex set $V$, where $|V|=N=\frac{c}{r-c-1}n$. For each $1 \le i \le c$ let $E_i$ be the set of those edges, that are contained in less than $\binom{r}{2}$ hyperedges of color $i$, and let $V_i$ be the set of vertices incident to at least one edge in $E_i$. Let $p=2c-r$ and for $v \in V$ let us define $$m(v):=|\{i : 1 \le i \le c, \ v \in V_i\}|.$$

\begin{clm}\label{upperdegree}
For every $v \in V$ we have $m(v) \le p+1$.
\end{clm}

\begin{proof}

By contradiction let us suppose that we have $m(v) \ge p+2$. Without loss of generality we can suppose that $v \in \bigcap_{i=1}^{p+2} V_i$. Then for $1 \le i \le p+2$ pick $e_i \in E_i$ such that each edge contains $v$ and for all $p+3 \le i \le c$ pick any $e_i \in E_i$. Then the cardinality of the vertex set of the endpoints of $\{e_i : 1 \le i \le c\}$ is at most $2c-p-1=r-1$. Let $\cH$ be the set of those hyperedges that contain every $e_i$ ($1 \le i \le c$). On the one hand the cardinality of $\cH$ is at least $N-r+1$, on the other hand $\cH$ contains at most $\binom{r}{2}$ hyperedges of each color, which contradicts our assumption on $n$.  
\end{proof}

By Claim \ref{upperdegree} we have $$\sum_{i=1}^{c}|V_i| \le (p+1)N.$$
This means that we have an $i$ with $1 \le i \le c$ such that $$|V_i| \le \frac{(p+1)N}{c}.$$

Which implies $|V\setminus V_i|\ge n$. All the edges inside $V\setminus V_i$ are contained in at least $\binom{r}{2}$ hyperedges of color $i$, thus Lemma \ref{philip} finishes the proof.

%So consider the vertex set $U:=V(K_N) \setminus V_i$. Note that $|U| \ge n$, and by Hall's condition we are done.

\smallskip

To prove $\textbf{(iv)}$ we follow the previous argument with a slight modification. Let $N = \lfloor\frac{c-1}{c-2}n\rfloor$. We will use the notation of the proof of \textbf{(iii)}. 

\begin{clm} There is a class $V_i$ with at most $\frac{N}{c-3/2}$ vertices.
%$\sum_{i=1}^{c}|V_i| \le \max_{1\le i\le c}|V_i|+N$.
\end{clm}

\begin{proof}
Observe first that if two different vertex classes, say $V_1$ and $V_2$ intersect, i.e., there are edges $e_1\in E_1$ and $e_2\in E_2$ sharing at least one vertex, then all the other classes are pairwise disjoint and also disjoint from the set $e_1 \cup e_2$. It is easy to see that otherwise we could find edges $e_3,e_4,...,e_c$ such that $e_i \in E_i$ ($1 \le i \le c$) and they are incident to at most $2c-2=r-1$ vertices. On the one hand, we have at least $N-r+1$ hyperedges that contain these edges, but on the other hand, among these hyperedges there can be at most $\binom{r}{2}$ hyperedges from each color class, which contradicts our assumption on $n$.

Let us consider an auxiliary graph $G$ on vertices $v_1,\dots,v_c$, where $v_i$ is connected to $v_j$ if and only $V_i$ intersects $V_j$. By the above argument, there are no independent edges in $G$, thus $G$ is either a star or a triangle (and potentially some isolated vertices). If $G$ is a star with center $v_1$, then every vertex of $V_1$ can be contained in at most one other $V_i$, which easily implies $\sum_{i=1}^{c}|V_i| \le |V_1|+N$. The statement follows for some $i\ge 2$.

If $G$ is a triangle with vertices $v_1,v_2,v_3$, then let $V'=V_1\cup V_2\cup V_3$. Observe that every vertex of $V'$ is contained in at most two of $V_1,V_2,V_3$, thus we have $2|V'|\ge |V_1|+|V_2|+|V_3|$. The other $V_i$s are disjoint from $V'$ and from each other. If the statement of the claim does not hold, we have the following contradiction. \[N=|V'|+|V\setminus V'|> |V'|+(c-3)\frac{N}{c-3/2}\ge \frac{|V_1|+|V_2|+|V_3|}{2}+N\frac{c-3}{c-3/2}>N\frac{c-3/2}{c-3/2}.\] 
\end{proof}

%note that if two different vertex classes, say $V_1$ and $V_2$ intersect, i.e., there are edges $e_1\in E_1$ and $e_2\in E_2$ sharing at least one vertex, then all the other classes are pairwise disjoint and also disjoint from the set $e_1 \cup e_2$. It is easy to see that otherwise we could find edges $e_3,e_4,...,e_c$ such that $e_i \in E_i$ ($1 \le i \le c$) and they are incident to at most $2c-2=r-1$ vertices. On the one hand, we have at least $N-r+1$ hyperedges that contain these edges, but on the other hand, among these hyperedges there can be at most $\binom{r}{2}$ hyperedges from each color class, which contradicts our assumption on $n$.  

%This implies that $\sum_{i=1}^{c}|V_i| \le |V_1|+N$, hence there is a class $V_i$ ($i \ge 2$) with at most $\frac{N}{c-1}$ vertices. This implies $$|V \setminus V_i| = N \frac{c-2}{c-1} \ge n.$$ at most

According to the above claim, $V_i$ has at most $\frac{N}{c-3/2}\le N-n$ vertices, thus $V\setminus V_i$ has at least $n$ vertices.

All the edges inside $V\setminus V_i$ are contained in at least $\binom{r}{2}$ hyperedges of color $i$, thus Lemma \ref{philip} finishes the proof again.
%So we have that there is a class $V_i$ with at most $\frac{n-2}{c-1}$ many vertices and we are done just like in the proof of \textbf{(iii)}. 
%The lower bound in $\textbf{(iv)}$ was proved in $\textbf{(i)}$, we just put it there to show how it compares to the upper bound.

\subsection{Proof of Proposition \ref{poly}}

We restate Proposition \ref{poly} below for convenience. We will use the well-known result of Erd\H os \cite{erdos} stating that for any $n$, there is a two-coloring of $K_n$ with no monochromatic clique of size $\lceil 2\log n\rceil$.

\begin{proposition*} 
$R^c(B^3K_n)= \Omega\left(\frac{n^{\lceil\frac{c}{2}\rceil}}{(\log n)^{\lceil c/2\rceil-1}}\right)$.
\end{proposition*}

\begin{proof} We prove the statement by induction on $c$. The cases $c=1$ and $c=2$ are trivial. Note that the case $c=3$ is the result of \cite{STZZ} mentioned in the introduction, and note that it is enough to prove the statement for $c$ odd. Indeed, obviously $R^{c+1}(B^3K_n)\ge R^c(B^3K_n)$, while the stated lower bound is the same for $c$ and $c+1$. Still, our proof below works for every $c\ge 3$.

Let us assume we know the statement for $c$. We are going to prove it for $c+2$. More precisely, by induction we can find a $c$-coloring of the complete $3$-uniform hypergraph on $h(n)=\Omega(\frac{n^{\lfloor\frac{c}{2}\rfloor+1}}{(\log n)^{\lfloor c/2\rfloor}})$ vertices without a monochromatic copy of $B^3K_n$. We will show a $(c+2)$-coloring of the complete $3$-uniform hypergraph $\cK_m^3$ on $m = h(n)\lfloor n/(2c\log n)\rfloor $ vertices without a monochromatic copy of $B^3K_n$.

We partition the vertex set of $\cK_m^3$ into $\lfloor n/2c\log n\rfloor$ parts, each of size $h(n)$, and for each part we color the hyperedges completely contained in that part using the colors $1,2, \ldots, c$, by induction. This way we have colored all the hyperedges completely inside a part without a monochromatic copy of $B^3K_n$ in colors $1,2, \ldots, c$ (as a $B^3K_n$ is connected, it would need to be contained inside a part). 

Additionally, for each part, we color all the pairs contained in the part (i.e., the complete $2$-uniform graph on $h(n)$ vertices) with colors $c+1$ and $c+2$ in such a way that the largest monochromatic complete ($2$-uniform) graph has order less than $2 \log h(n) < 2c\log n$. (Here we used a well-known construction of Erd\H os.)

Each hyperedge is either completely contained in one of the parts (in which case we have already colored it), or it intersects exactly two of the parts or it intersects three of the parts. If a hyperedge contains the pair $u,v$ from a part and a third vertex from a different part, then we color it by the color of the pair $uv$ (thus it either gets the color $c+1$ or $c+2$). If a hyperedge intersects three parts, we color it by $c+2$. This completes the coloring of all the hyperedges. Let us assume for a contradiction that there is a monochromatic copy of $B^3K_n$. As argued before, this copy cannot be in any of the colors $1,2, \ldots, c$, so it must be of the color $c+1$ or $c+2$. Moreover, by the pigeon-hole principle, this monochromatic copy contains a set $S$ of at least $2c\log n$ vertices from one of the parts (as the number of parts is $\lfloor n/2c\log n\rfloor$). Then, by construction, all of the pairs in $S$ must have been colored by the same color (either $c+1$ or $c+2$). This contradicts the discussion in the previous paragraph, finishing the proof.
%Then a Berge clique of color $c+1$ or $c+2$ intersects every part in a clique of the same color, thus in less than $2c\log n$ vertices. Thus the total number of vertices in a Berge-clique of color $c+1$ or $c+2$ is less than $n$, finishing the proof.
\end{proof}

\subsection{Proof of Proposition \ref{3-uniform}}

We restate Proposition \ref{3-uniform} below for convenience.

\begin{proposition*} If $n\ge m>1$ and $n+m\ge 7$, then $n+m-3\le R(B^3K_n,B^3K_m)\le n+m-2$.
% We have
% \begin{enumerate}
% \item [\textbf{(i)}] 

% \item [\textbf{(ii)}] If $n\ge c$, then $R^c(B^3K_n)> (c-1)(n-2)$.
% \end{enumerate}
\end{proposition*}

For the upper bound in we use induction on $n+m$. The statement is trivial for $m=2$. The other base case $n=4$, $m=3$ can be proved by a simple case analysis. Consider a $2$-colored complete $3$-graph on $n+m-3$ vertices, and set aside a vertex $u$. By induction there are both a blue $B^3K_{n-1}$ with vertex set $A$ and a red $B^3K_{m-1}$ with vertex set $B$ on the remaining vertices. We will show that there is either a blue $B^3K_n$ on $A\cup \{u\}$ or a red $B^3K_m$ on $B\cup \{u\}$.

Consider the link graph of $u$. This is a two-colored complete graph, thus there is a monochromatic spanning tree in it, say, a blue one. Pick a vertex $v\not\in A$ and apply Lemma \ref{feszfa}. Then every vertex $w\in A$ can be connected to $u$ using the blue hyperedge containing $u$ and $f(w)$. These hyperedges are distinct from each other as $f$ is a bijection, and distinct from the hyperedges that form the blue Berge-$K_{n-1}$ on $A$, as those do not contain $u$.

\smallskip

For the lower bound, we take an $(m-2)$-set $U$ and an $(n-2)$-set $U'$. Every 3-edge $H$ shares at least two vertices with either $U$, in which case we color $H$ red, or with $U'$, in which case we color $H$ blue. A blue Berge-$K_n$ contains two vertices from $U$, but those cannot be connected with an edge contained in a blue hyperedge, a contradiction. A red Berge-$K_m$ leads to contradiction similarly.

\subsection{Berge trees. Proof of Theorem \ref{trees}}

The next lemma deals with small trees. It will serve as the base case of induction later in the proof of Theorem \ref{nagyfa}. Lemma \ref{kisfa} combined with Theorem \ref{nagyfa} gives Theorem \ref{trees}.

% \begin{lemma}\label{csillag} If $n \ge 6$, then $R(B^3S_n,B^3S_n)=n$.

% \end{lemma}

% \begin{proof} Let us consider a blue-red complete $3$-graph on $n$ vertices, and one of its vertices $v$. Consider the link graph $L_v$ with the inherited coloring. Let $G$ be the graph of blue edges in $L_v$ and $G'$ be the graph of red edges. Then one of them, say $G$ is connected. Let us assume first $G$ is not a tree. Then we remove some edges to obtain a tree and let $u$ be a vertex incident to a removed edge $uu'$. By Lemma \ref{feszfa} we can find distinct edges $f(w)$ for every vertex $w\neq u$ in $G$. Those edges together with $v$ and the hyperedge $uu'v$ form a blue Berge-$S_n$.

% Let us assume now $G$ is a tree, but not a star. Then $G'$ is also connected, and $G'$ is not a tree as $n\ge 6$, thus we can similarly find a red Berge-$S_n$.

% Hence we can assume $G$ is a star with center $u$. In fact we can assume that the link graph of every vertex consists of a monochromatic star $S_{n-1}$ and its complement in the other color. In particular, in $L_u$ $v$ has blue degree $n-2$, thus no other edge can be blue. Finally, consider a third vertex $w$ and its link graph $L_w$. There the edge $uv$ is blue, but every other vertex is connected to both $u$ and $v$ by a red edge. Hence the red graph is connected but not a star, a contradiction.

% %Then we look at the link graph of the center of the star $v'$, and then of another vertex if needed. 

% \end{proof}

\begin{lemma}\label{kisfa}
Suppose $T_i, T_i'$ are two trees on $i$ vertices. Then we have $R(B^3T_2,B^3T'_2)=3$, $R(B^3T_3,B^3T'_3)=4$ and $R(B^3T_4,B^3T'_4)=5$. $R(B^3S_5,B^3S_5)=6$ and if at least one of $T_5$ and $T_5'$ is not the star, then $R(B^3T_5,B^3T'_5)=5$. If $1<j<i\le 5$, then we have $R(B^3T_i,B^3T'_j)=i$.
\end{lemma}

\begin{proof} It is obvious that if $j\le i$, then to find a monoblue $T_i$ or a monored $T_j$ we need at least $i$ vertices and at least $i+j-3$ hyperedges. These imply all the lower bounds except for $R(B^3S_5,B^3S_5)>5$. To show that, let us take five vertices $v_1,\dots, v_5$ and color the hyperedges of the form $v_iv_{i+1}v_{i+2}$ (where the addition is modulo $5$) blue, and the remaining hyperedges red. Then every vertex is incident to exactly three blue and three red hyperedges, thus there is no monochromatic Berge-$S_5$.

Let us prove now the upper bounds. It is obvious that $R(B^3T_2,B^3T'_2)\le 3$ and \\ $R(B^3T_3,B^3T'_3)\le 4$. To show $R(B^3T_4,B^3T'_4)=5$, observe that an arbitrary vertex $v$ is contained in three hyperedges of the same color, say blue. It is easy to see that those three hyperedges form a Berge-$S_4$ and a Berge-$P_4$ as well.

Let us assume now $T_5$ is not a star and consider $R(B^3T_5,B^3T'_5)$. Observe that if four hyperedges contain the same vertex, and we have five vertices altogether, they form not only $S_5$ but also any of the other $5$-vertex trees. Thus the only remaining case is when every vertex is incident to exactly three blue and three red hyperedges. It is easy to check that we can find a blue $T_5$ in this case.

If we have six vertices, any vertex is contained in at least four hyperedges in one of the colors. It is easy to see that those four hyperedges form a Berge-$S_5$, showing $R(B^3S_5,B^3S_5)=6$.

%The above statements and proofs easily imply $R(B^3T_i,B^3T'_j)=i$ if $1<j<i\le 5$.

It is left to deal with the case $R(B^3T_i,B^3T'_j)=i$ if $1<j<i\le 5$. The lower bound follows by coloring all the hyperedges red on less than $i$ vertices. For the upper bound, in case $j=2$, there cannot be any red hyperedges, while in case $j=3$, there can be at most one red hyperedges. In both cases one can easily find the blue $B^3T_i$. In case $j=4$ we have $i=5$ and if there is no blue $B^3T_5$ on five vertices, then we can find a red $B^3T'_5$ for any $T_5'\neq S_5$. At least one of those contains $T'_4$, finishing the proof.
\end{proof}

\begin{thm}\label{nagyfa} Let $T_1$ and $T_2$ be trees with $6\le n=|V(T_1)|\ge |V(T_2)|$. Then $$R(B^3T_1,B^3T_2)=n.$$

\end{thm}

\begin{proof} We apply induction on $n$ and also assume indirectly that we are given a blue-red $\cK_n^3$ that does not contain a monoblue $T_1$, nor a monored $T_2$. We dealt with the base cases $n\le 5$ in Lemma \ref{kisfa}, thus we assume $n\ge 6$. We will use multiple times Lemma \ref{philip2} with $r=3$ and $t=2$. Both $T_1$ and $T_2$ are triangle-free, thus it implies that we cannot find a copy of $T_1$ such that each of its edges are contained in at least two blue hyperedges, and similarly we cannot find a copy of $T_2$ such that each of its edges are contained in at least two red hyperedges.
%, and the cases when one or both of $T_1$ and $T_2$ are $S_6$ or $P_6$. We will deal with these cases later. 

Let us introduce some definitions.
We call an edge $uv$ \textit{deep red} if all hyperedges containing both $u$ and $v$ are red, and \textit{deep blue} if all hyperedges containing both $u$ and $v$ are blue. 
Note that a vertex cannot be incident to both a deep blue and a deep red edge, as there is a hyperedge containing those two edges.

\ 

\textbf{Case 1.} There is a vertex $v$ in the blue-red $\cK_n^3$ that is not incident to  a deep blue, nor to a deep red edge. If $T_1=T_2=S_6$, then $n=6$ and then $v$ is contained in at least five hyperedges of one of the colors, say blue. Every other vertex is contained in at least one of those blue hyperedges, otherwise $v$ is incident to a deep red edge. Let $E_0$ be the set of edges incident to $v$ and blue be color 1. Consider $\Gamma_1(E_0)$, a matching in it covering $A$ would finish the proof. Otherwise there is a subset $A'$ of $A$ with less than $|A'|$ neighbors in $\Gamma_1(E_0)$ by Hall's theorem. Obviously $5>|A'|>1$, but than $A\setminus A'$ has $1\le i\le 3$ vertices. This means $v$ and a set of its $i$ neighbors have at least $i+1$ 3-edges that each contain $v$, which is impossible. 

%Then it is easy to see we can find a blue Berge-$S_6$ with center $v$ using Hall's theorem. 

Let us assume now that $T_1$ and $T_2$ are not both $S_6$. Let $T_i'$ be a tree that is obtained from $T_i$ by deleting a leaf such that at least one of $T_1'$ and $T_2'$ is not $S_5$.

Let us consider a blue-red $\cK_n^3$ and assume it contains neither a blue Berge-$T_1$ nor a red Berge-$T_2$. Let us delete an arbitrary vertex $v$. Then by induction we either find a blue Berge-$T_1'$ or a red Berge-$T_2'$ in the blue-red $\cK_{n-1}^3$ we obtained by deleting $v$. Let us assume it is a blue Berge-$T_1'$, then there is a vertex $u$ such that the edge $uv$ would extend $T_1'$ to $T_1$. Note that the hyperedges forming that blue Berge-$T_1'$ do not contain $v$. As $uv$ is not deep red, we can find a blue hyperedge $H$ containing both $u$ and $v$. Adding that to the blue Berge-$T_1'$ we obtain a blue Berge-$T_1$, as $H$ is distinct from the hyperedges in the the blue Berge-$T_1'$. In case we found a red Berge-$T_2'$, we proceed similarly. 

%If there is any blue hyperedge containing $uv$, we are done, as that hyperedge is unused, because we found a blue $T_1'$ in the complete $3$-graph that does not even contain $v$. So we obtained that all hyperedges containing both $u$ and $v$ are red.

\ 

\textbf{Case 2.} Every vertex is incident to a deep blue edge, or every vertex is incident to a deep red edge. We will assume every vertex is incident to a deep blue edge, the other case follows similarly, as we do not use $|V(T_1)|\ge |V(T_2)|$. Consider an arbitrary edge $uv$. If it is deep blue, it is contained in at least $n-2$ blue hyperedges. If not, then there are $u',v'$ such that $uu'$ and $vv'$ are deep blue by our assumption. Note that if $u'\neq v'$, then $uv$ is contained in at least two blue hyperedges: $\{u,u'v\}$ and $\{u,v,v'\}$. 

\textbf{Case 2.1.} Every vertex is incident to exactly one deep blue edge. Then we take a copy of $T_1$ in the shadow graph. 
%Consider an arbitrary edge $uv$ of it. 
%It is either a deep blue edge, in which case it is contained in $n-2$ blue hyperedges, or $uu'$ and $vv'$ are deep blue edges, where the four vertices are pairwise distinct. In this case $v$ is contained in the blue hyperedges $uuvu'$ and $uvv'$. Thus 
Every edge of that copy of $T_1$ is contained in at least two blue hyperedges, and then we can find a blue Berge-$T_1$ using Lemma \ref{philip2}, a contradiction. 

\textbf{Case 2.2.} There are two adjacent deep blue edges. We take a copy of $T_1$ in the shadow graph that contains them. Let $E_0$ be the set of edges of that copy of $T_1$. We consider $\Gamma_1(E_0)$, where color 1 is blue. A matching covering $A$ would lead to a contradiction, thus there is a set $A'\subset A$ with $|N(A')|<|A'|$ by Hall's theorem. We pick a minimal $A'$ with this property. Then $A'$ has to contain a vertex $a=uv$ of degree $1$ in $\Gamma_1(E_0)$, otherwise the number of edges in $\Gamma_1(E_0)$ between $A'$ and $N(A')$ is at least $2|A'|$ and at most $2|N(A')|$ (since vertices of $B$ have degree at most 2), a contradiction. Thus $a\in A'$ is connected to only one vertex $b=uvw\in B$. If no other element of $A'$ is connected to $b$, then $A'\setminus \{a\}$ is a smaller blocking set, a contradiction. Thus another subedge, say $uw$ of the hyperedge is in $A'$. Recall that if $uv$ is contained in only one blue hyperedge $uvw$, then the deep blue edge incident to $u$ is $uw$, and the deep blue edge incident to $v$ is $vw$. Hence $uw$ is contained in $n-2$ blue hyperedges. This means $|N(A')|\ge n-2$, but as $A$ corresponds to a tree, we have $n-2 \le |N(A')|<|A'|\le |A|\le n-1$. This implies $A'=A$. But as the copy of $T_1$ we took contains two deep blue edges, we have $|N(A)|\ge 2n-5$, a contradiction.

\ 

\textbf{Case 3.} $T_1$ or $T_2$ is a star.

Let us assume first that $T_1$ is a star and let $uv$ be a deep blue edge. If there is a blue hyperedge $vwz$ with $w\neq u\neq z$, then we can find a monoblue Berge-star with center $v$. Indeed, for an edge $vy$ with $u\neq y\neq w$ we use the blue hyperedge $uvy$, for the edge $vw$ we use the blue hyperedge $vwz$, while for the edge $vu$ we use the blue hyperedge $uvw$.

If, on the other hand, every hyperedge containing exactly one of $u$ and $v$ is red, then every edge other than $uv$ is contained in at least two red hyperedges. We pick a copy of $T_2$ not containing the edge $uv$, and Lemma \ref{philip2} finishes the proof. If $T_2$ is a star, the proof follows similarly, as we did not use $|V(T_1)|\ge |V(T_2)|$.

\ 

\textbf{Case 4.} There are deep blue and deep red edges, and $T_1$ and $T_2$ are both non-stars.

\textbf{Case 4.1.} $T_1$ or $T_2$ satisfies \textbf{(ii)} of Lemma \ref{ujfa}. We assume it is $T_1$, the other case follows similarly. Note that often there are multiple ways to choose the two edges of $T_1$ specified in \textbf{(ii)} of Lemma \ref{ujfa}. We pick them arbitrarily, except in case $T_1$ is a star with center $v$ and another edge $uw$ is attached to a leaf $u$, and we have to pick two adjacent edges. In that case we pick the edge $uv$ and an edge $vz$ for an arbitrary $z$.

\textbf{Case 4.1.1.} There is a deep blue edge $e_1=uv$, and another edge $e_2=wz$ is contained in at most one red hyperedge $wzx$. Note that $e_1$ and $e_2$ may share a vertex, or $x$ can be one of $u$ or $v$. The edges $e_1$ and $e_2$ will correspond to the two independent or adjacent edges described in \textbf{(ii)} of Lemma \ref{ujfa}, depending on if $e_1$ and $e_2$ are independent. Let $e_1'$ and $e_2'$ be those two edges of $T_1$. As $T_1$ is not a star and has at least six vertices, there is a vertex $x'$ in it that is connected to an endpoint of only one of $e_1'$ and $e_2'$, say $e_1'$. Then we let $e_1$ correspond to $e_1'$ and $e_2$ correspond to $e_2'$. 
Let $x$ (if exists and is different from $u$ and $v$) correspond to $x'$. 
%and connect it to the appropriate endpoint of $e_1$ using the edge $uvx$. 
For every other vertex of $T_1$, we identify it with an arbitrary other vertex in $\cK_n^3$. This way we choose a copy of $T_1$ in the shadow graph.

Now we are going to build a bijection $g$ from the edges of this copy of $T_1$ to the blue hyperedges containing them. We know that if $x$ exists and is different from $u$ and $v$, then either $xv$ or $xu$ is an edge $e_0$ of this copy of $T_1$,
%by \textbf{(ii)} of Lemma \ref{ujfa}, 
and we let $g(e_0)=uvx$.
For another edge $e$ in $T_1$, one of its endpoints $y$ is in $e_i$ ($i\le 2$). If the other endpoint $y'$ is not in $e_1\cup e_2$, then let $g(e):=e_i\cup \{y'\}$. If $y$ is in both $e_1$ and $e_2$, this does not give a complete definition; in that case we let $g(e)=e_1\cup\{y'\}$. If $y$ is in $e_1$ and $y'$ is in $e_2$ (but $e_1\neq e\neq e_2$), then let $g(e):=e_1\cup \{y'\}$. 

Observe that we picked distinct hyperedges so far. Indeed, if $e_i\cup \{y'\}$ is picked twice for some $i$, then $y'$ is connected to both endpoints of $e_i$, thus there is a triangle in $T_1$, a contradiction. Furthermore, we picked only blue hyperedges so far, as each of them contains $e_1$ or $e_2$, and $wzx=e_2\cup \{x\}$ has not been picked (as $x$ is not connected to an endpoint of $e_2$).

We still have to pick $g(e_1)$ and $g(e_2)$. At this point $n-3$ blue hyperedges have been picked as $g(e)$ for some edge $e$ of the copy of $T_1$. The edge $e_2$ is contained in either $n-3$ blue hyperedges (and in this case we have picked $g(e_0)=uvx$ which does not contain it) or $n-2$ blue hyperedges. In both cases there is a blue hyperedge containing $e_2$ that is not $g(e)$ for any $e \in E(T_1)$ so far, let $g(e_2)$ be that hyperedge. Finally, $e_1$ is contained in $n-2$ blue hyperedges, thus we are done, unless $e_1$ is contained in $g(e)$ for every edge $e$ of our copy of $T_1$. 

In particular this means that $g(e_2)$ contains $e_1$, thus $e_2$ shares a vertex $u=w$ with $e_1$. Also, every other vertex in the copy of $T_1$ is connected to $u$ or $v$. If $x$ exists, it cannot be connected to $u$ by the definition of $x'$. If there is a vertex $y$ with $v\ne y\neq z$ connected to $u$, then we change $g$. We let $g(yu)=uzy$ instead of $uvy$, and let $g(e_1)=uvy$. Observe that $uzy$ is blue, since it contains $e_2$ but not $x$. Furthermore, $uzy$ was not $g(e)$ for any edge $e$ of the copy of $T_1$ originally, as $y$ is connected only to $u$, and $g(uy)$ was $e_1\cup \{y\}$. Thus we found distinct blue hyperedges containing each edges of a copy of $T_1$, thus a monoblue Berge-$T_1$, a contradiction. 
%could pick for the edge $uy$ the hyperedge $uzy$, and then for the edge $e_1$ we can use $uvy$. 

Finally, if there is no vertex $y$ connected to $u$ (besides $v$ and $z$), then $T_1$ consists of a star with center $v$ and one additional edge $uw$ attached to a neighbor $u$ of $v$. This is the case we specified at the beginning of Case 4.1, in this case we would not pick $uw$ as $e_2'$, a contradiction. 

%We found distinct blue hyperedges containing each edges of a copy of $T_1$, thus a monoblue Berge-$T_1$, a contradiction. 

\textbf{Case 4.1.2.} There is a deep blue edge, and every other edge is contained in at least two red hyperedges. Then Lemma \ref{philip2} applied to the hypergraph consisting of the red hyperedges finishes the proof.

\textbf{Case 4.2.} $T_1=T_2=P_6$. 
%Then $n=6$ and for one of the colors, say red, there are at least two deep red edges. We can assume $xz$ is a deep red edge and a third vertex $y$ is also incident to a deep red edge. 
We follow the argument in Case 4.1. 
$P_6$ does not satisfy \textbf{(ii)} of Lemma \ref{ujfa}, because although we can find two independent edges in it such that every other edge shares a vertex with one of them, we cannot find two adjacent ones. Thus we can apply the proof of Case 4.1.1 in case the edges $e_1$ and $e_2$ are independent, and apply the proof of Case 4.1.2 always. It means that if if $uv$ is a deep blue edge, then every edge independent from it is contained in at least two red hyperedges. Let $xz$ be a deep red edge.
%Let us assume the edge $xy$ is not deep red, but $xz$ is. 

Then we can embed $P_6$ in the order $uxzwyv$, where $w$ and $y$ are the other two vertices. The edges of this $P_6$ are contained in at least $1$, $4$, $2$, $2$ and $1$ red hyperedges in this order. Now we are going to build a bijection $g$ from the edges of this copy of $P_6$ to the red hyperedges containing them. First we choose $g(ux)$ and $g(yv)$ arbitrarily from the red hyperedges containing them. Then we pick $g(wy)$. It is possible that one of the red hyperedges containing $wy$ is $g(yz)$, but we still have another one that we can pick as $g(wy)$. Similarly we have an unused red hyperedge to pick as $g(zw)$ and at least two choices for $g(xz)$.
%One can easily check that we can find a red Berge copy of that $P_6$. Indeed, we build 

%The remaining case is when $xy$, $xz$ and $yz$ are all deep red. Then every edge containing at least one of these three vertices is contained in at least two red hyperedges. Let us embed $P_6$ in the order $xuyvzw$. Then every edge of it is contained in at least two red hyperedges, thus Lemma \ref{philip2} finishes the proof.

\textbf{Case 4.3.} $T_1$ and $T_2$ both satisfy \textbf{(i)} of Lemma \ref{ujfa}. Thus we can delete $v_i$ and the neighboring leaves from $T_i$ to obtain $T_i'$, for $i=1,2$. Let $n'$ be the number of vertices of the larger of the trees $T_1'$ and $T_2'$. Then we delete $n-n'$ vertices from the blue-red $\cK_n^3$ such that one of them $x$ is incident to a deep blue edge $xx'$ and another one $y$ is incident to a deep red edge $yy'$, and we do not delete $x'$ nor $y'$. Let $Q$ be the set of $n'$ vertices we do not delete.

If$n'\ge 6$, then we can apply induction and find, say, a monoblue Berge-$T_1'$ on the remaining vertices. That means we fins a copy of $T_1'$ in the shadow graph of the blue-red $\cK_{n'}^3$ with vertex set $Q$, such that there are distinct blue hyperedges containing its edges. In case $n'=5$, then $T_1'$ and $T_2~$ are not stars, thus we can similarly find a monochromatic Berge-$T_1'$ by Lemma \ref{kisfa}.

The copy of $T_1'$ found this way has a vertex $z$ corresponding to the vertex $u$ in \textbf{(i)} of Lemma \ref{ujfa}, i.e. adding an edge $zx$ and the edges $xx_1,xx_2,\dots, xx_k$ would result in a copy of $T_1$, where the $x_i$'s are all the vertices of the original blue-red $K_n^3$ that are not in $Q$. For these edges, we consider the hyperedges $zxx'$ and $xx'x_i$ (if $z\neq x'\neq x_i$ for every $i$). These are all blue as $xx'$ is deep blue, and they are obviously distinct from each other (they are also obviously distinct from the blue hyperedges used earlier, as they use only vertices from $Q$).

In case $x'=z$ or $x'=x_i$ for some $i$, the edge $xx'$ is not matched with a blue hyperedge. In that case we can choose an arbitrary vertex $w$ of $Q$ and for the edge $xx'$ we use the hyperedge $xx'w$. This new hyperedge is distinct from the others we picked, as the ones we picked first are inside $Q$, while the ones we picked later share only $x$ with $Q$. The hyperedge $xx'w$ is blue as $xx'$ is deep blue, giving us a monoblue Berge-$T_1$, a contradiction. In case we find a monored Berge-$T_2'$ by induction, we proceed similarly.

\textbf{Case 4.4.} $T_1$ satisfies \textbf{(i)} but not \textbf{(ii)} of Lemma \ref{ujfa}, and $T_2=P_6$. In this case we find $T_1'$ as in the previous case and let $T_2'=T_2$. We proceed as in Case 4.3. Note that by the assumption on $T_1$, we have $n>6$, which implies $n>n'$, thus we can use induction. From this point the proof is exactly the same as in Case 4.3, except that if the induction finds a monored $T_2'$, we are already done without any further steps. 

\end{proof}

\subsection{Proof of Proposition \ref{vercov}}

%In case of $c=2$ and $r=3$ from Proposition \ref{3-uniform} we know that the Berge Ramsey number is between roughly $n$ and $2n$ for a graph on $n$ vertices. Here we show that it is close to $n$ for graphs with very small vertex cover number, but close to $2n$ for very large vertex cover number.

We restate Proposition \ref{vercov} below for convenience.

\begin{proposition*} If $G$ is a graph on $n$ vertices with vertex cover number $k\ge 3$, then $$2k-1\le R(B^3G,B^3G)\le n-k+2R(K_k,K_k).$$

\end{proposition*}

\begin{proof} 
For the upper bound, we set aside a set $A$ of $2R(K_k,K_k)-1$ vertices, let $B$ be the set of remaining vertices. Consider all the link graphs of the vertices in $A$, restricted to $B$, with the inherited coloring. These link graphs each have a monochromatic spanning tree, thus at least $R(K_k,K_k)$ of them have, say, a blue spanning tree. Let $A'\subset A$ be a subset of size $R(K_k,K_k)$, such that the link graph of each of its elements has a blue spanning tree. Let us pick $v\in B$ and apply Lemma \ref{feszfa} to the blue spanning trees of these link graphs. This way for every vertex in $B':=B\setminus \{v\}$ we found $R(K_k,K_k)$ blue edges containing them, such that they, extended with the corresponding vertex of $A'$ give a blue Berge-$K_{R(K_k,K_k),n-k}$. 

Let us color an edge inside $A'$ blue if it is contained in at least three blue hyperedges, and red otherwise. Note that this definition is not symmetric. By the definition of the Ramsey number we can find a monochromatic clique of size $k$ with this coloring. Assume it is blue. By Lemma \ref{philip} this gives a blue Berge-$K_k$ on those vertices. Deleting the other vertices of $A'$ we obtain a blue Berge copy of a graph on $n$ vertices with $k$ vertices having degree $n-1$. This obviously contains $G$. Note that the hyperedges used in the first part of the proof contained exactly one vertex from $A$, while the hyperedges used in the second part contain at least two, thus they are distinct.

Hence we can assume we obtain a red Berge-$K_k$ on the vertex set $A''=\{v_1,\dots,v_k\}$. It means at most $2\binom{k}{2}$ other vertices can be contained in a red hyperedge together with any of the edges inside $A''$. Let us delete those vertices and consider the set $B''$ of the remaining at least $n-k+1$ vertices, let $u$ be one of them. For two vertices $v_i,v_j\in A'$ we take the red hyperedge $\{u,v_i,v_j\}$. For vertices $v_i\in A'$, $w\in B''\setminus\{u\}$ we take the red hyperedge $\{v_i,v_{i+1},w\}$, where $i+1$ is taken modulo $k$. It is easy to see that we took distinct hyperedges, and they form a red Berge copy of a graph on $n$ vertices with $k$ vertices having degree $n-1$. 

%Let us consider a set $A$ of $R(k,k)$ vertices and color an edge inside $A$ blue if it is contained in at least three blue hyperedges, and red otherwise. Notice that every edge gets at least one color. By the definition of the Ramsey number we can find a monochromatic, say blue clique of size $k$ with this coloring on the vertex set $A'=\{v_1,\dots,v_k\}$. It means at most $2\binom{k}{2}$ other vertices can be contained in a red hyperedge together with any of the edges inside $A'$. Let us delete those vertices and consider the set $B$ of the remaining $n-k+1$ vertices, let $u$ be one of them. For two vertices $v_i,v_j\in A'$ we take the blue hyperedge $\{u,v_i,v_j\}$. For vertices $v_i\in A'$, $w\in B$ we take blue hyperedges of the form $\{v_i,v_{i+1},w\}$, where $i+1$ is taken modulo $k$. It is easy to see that we took distinct hyperedges, and they form a blue Berge copy of a graph on $n$ vertices with $k$ vertices having degree $n-1$. This obviously contains $G$.

%half of them, a set $A$ has blue spanning tree in the link graph to the remaining vertices, thus we have blue $K_{2^k,n-k}$. Then inside $A$ we color an edge blue if it is in at least 2 blue hyperedges and red otherwise. We find a mono clique of size $k$, if it is blue, we are done, if it is red, they are in almost monored hyperedges, we just have to delete one vertex from the remaining graph for each of the $\binom{k}{2}$ edges, and the remaining graph is mono-red Berge-$K_{k,n-k}$.

For the lower bound, we consider a complete $3$-graph on $2k-2$ vertices and partition it into two parts $A$ and $B$ of size $k-1$. Every hyperedge intersects one of the parts in at least two vertices. We color it red if it is part $A$ and blue otherwise. Then any red edge is incident to one of the $k-1$ vertices in part $A$ and every blue edge is incident to one of the $k-1$ vertices in part $B$, thus any monochromatic graph has vertex cover number less than $k$.
\end{proof}

\section{Concluding remarks}

Let us note that every Berge copy of a connected graph is a connected hypergraph. Therefore, an upper bound on the size of the largest monochromatic component of every $c$-colored $\cK_n^r$ can give lower bounds in Berge Ramsey problems. There are some strong results for different values of $c$ and $r$, see \cite{furgya,gyarfas,gyh}. For example, there is a $4$-colored complete $3$-graph such that the largest monochromatic component has size $3n/4+o(n)$ by \cite{gyarfas}. This implies that for any connected graph $G$ on $n$ vertices, we have $R^4(B^3G)\ge 4n/3-o(n)$. However, this does not help us obtain a lower bound in case $c=r=3$, as in that case there is always a monochromatic spanning subhypergraph.

\section*{Acknowledgements}

We thank the anonymous reviewers for their careful reading of our manuscript and their many insightful comments and suggestions improving the presentation of our article.

\ 

Research of Gerbner was supported by the J\'anos Bolyai Research Fellowship of the Hungarian Academy of Sciences and by the National Research, Development and Innovation Office -- NKFIH, grant K 116769, KH 130371 and SNN 129364.

\noindent
Research of Methuku was supported by the National Research, Development and Innovation Office -- NKFIH, grant K 116769.

\noindent
Research of Vizer was supported by the National Research, Development and Innovation Office -- NKFIH, grant SNN 116095, 129364, KH 130371 and by the J\'anos Bolyai Research Fellowship of the Hungarian Academy of Sciences.

\end{document}